\documentclass[12pt]{article}
\usepackage{amsmath,amsthm,mathrsfs}
\usepackage[english]{babel}

\usepackage{graphicx}
\usepackage{amssymb, xcolor}
\usepackage{epstopdf, enumerate}
\usepackage{amsfonts,amssymb,amsthm}
\usepackage[latin1]{inputenc}
\usepackage{pifont, url, lscape}
\usepackage{amsfonts,amstext,amssymb,verbatim,epsfig}
\usepackage{pgf,tikz}
\usepackage{float}
\usetikzlibrary{arrows}
\usepackage{hyperref}
\usepackage{cite}
\usepackage{epstopdf}
\usepackage{color}
\usepackage{transparent}
\usepackage{subfigure}

\usepackage[left=1in,top=1in,right=1in,bottom=1in]{geometry}

\usepackage{xcolor}
\usepackage[normalem]{ulem}

\hypersetup{colorlinks=true,pdfborder=000,  citecolor  = blue, citebordercolor= {magenta}}
\hypersetup{linkcolor=purple}
\makeatletter
\let\reftagform@=\tagform@
\def\tagform@#1{\maketag@@@{(\ignorespaces\textcolor{purple}{#1}\unskip\@@italiccorr)}}
\renewcommand{\eqref}[1]{\textup{\reftagform@{\ref{#1}}}}
\makeatother

\newtheorem{theorem}{Theorem}[section]

\newtheorem{corollary}[theorem]{Corollary}
\newtheorem{lem}[theorem]{Lemma}
\newtheorem{prop}[theorem]{Proposition}
\newtheorem{thm}[theorem]{Theorem}

\theoremstyle{remark}
\newtheorem{rem}[theorem]{Remark}

\usepackage{mathrsfs}
\usepackage{verbatim}

\newcommand{\E}{\ensuremath{\mathbb{E}}}
\newcommand{\Pro}{\ensuremath{\mathbb{P}}}
\newcommand{\indi}{\ensuremath{\textbf{1}}}

\renewcommand{\epsilon}{\varepsilon}

\newcommand{\mathdef}{\mathrel{\mathop:}=}

\newcommand{\br}{\vspace{1mm}\\}
\newcommand{\Br}{\vspace{2mm}\\}

\def\<{\langle}
\def\>{\rangle}

\numberwithin{equation}{section}

\usepackage{tikz}
\usetikzlibrary{arrows,decorations.pathmorphing,backgrounds,fit,positioning,shapes.symbols,chains}
\tikzset{>=latex}
\usetikzlibrary{decorations.markings}

\date{\today} 
\begin{document}

\title{On the time constant of high dimensional \\ first passage percolation}
\author{Antonio Auffinger \thanks{ tuca@northwestern.edu. The research of A. A. is supported by NSF grant DMS-1597864.}  \\ \small{Northwestern University} \and Si Tang  \thanks{sitang@galton.uchicago.edu} \\ \small{University of Chicago}}\maketitle


\begin{abstract}
We study the  time constant $\mu(e_{1})$ in first passage percolation on $\mathbb Z^{d}$ as a function of the dimension. We prove that if the passage times have finite mean,
$$\lim_{d \to \infty} \frac{\mu(e_{1}) d}{\log d}  = \frac{1}{2a},$$
where $a \in [0,\infty]$ is a constant that depends only on the behavior of the distribution of the passage times at $0$.
For the same class of distributions, we also prove that the limit shape is not an Euclidean ball, nor a $d$-dimensional cube or diamond, provided that $d$ is large enough.
\end{abstract}
\maketitle
\section{Introduction and main results}
We study first passage percolation on $\mathbb Z^{d}$ for $d$ large. The model is defined as follows. We place a non-negative random variable $\tau_e$, called the passage time of the edge $e$, at each nearest-neighbor edge in $\mathbb Z^d$. The collection $(\tau_e)$ is assumed to be independent, identically distributed with common distribution $F$. 

 A {\it path} $\gamma$ is a finite or infinite sequence of nearest neighbor edges in $\mathbb Z^d$ such that each two consecutive edges in the sequence intersect. For any finite path $\gamma$ we define the passage time of $\gamma$ to be
\[
T(\gamma)=\sum_{e \in \gamma} \tau_e.
\]
Given two points $x,y \in \mathbb{Z}^d$ one then sets
\begin{equation*}\label{definition:passagetime}
T(x,y) = \inf_{\gamma} T(\gamma),
\end{equation*}
where the infimum is over all finite paths $\gamma$ that start at the point $x$ and end at $y$. For a review and the current state of the art of the model, we invite the readers to see the recent notes \cite{FPPSurvey} or the classical paper of Kesten \cite{Aspects}.

Here we focus on the large $d$ behavior of the time constant and limit shape of the model. These are defined as follows. Let $e_{1}, \ldots, e_{d}$ be the coordinate vectors of $\mathbb Z^{d}$.

The time constant $\mu(e_1) \in [0,\infty)$ is defined as 
\begin{equation*}
 \mu(e_1) = \lim_{n \rightarrow \infty} \frac{T(0,ne_1)}{n} \quad \text{ a.s. and in } L^1.
\end{equation*}
If $\E \tau_{e} < \infty$, $\mu(e_{1})$ exists. See \cite[Theorem 2.1]{FPPSurvey} and the discussion therein. 

For each $t \geq 0$ let 
\[
B(t) = \{y \in \mathbb{R}^d~:~T(0,[y])\leq t\} ,
\]
where $[y]$ is the unique point in $\mathbb Z^{d}$ such that $y\in [y] + [0,1)^{d}$. The pair $(\mathbb{Z}^d,T(\cdot, \cdot))$ is a pseudo-metric space and $B(t) \cap \mathbb {Z}^d$ is the (random) ball of radius $t$ around the origin. 
The limit shape is defined by the famous shape theorem as follows.

Assume
\begin{equation}\label{eq:conditionmomentsLS}
\E \min \{ t_1^d, \ldots, t_{2d}^d\} <\infty,
\end{equation}
where $t_i, i=1, \ldots 2d$, are independent copies of $\tau_e$ and  
\begin{equation}\label{eq:conditionmomentsLS2}
F(0)<p_c(d),
\end{equation} where $p_c(d)$ is the threshold for bond percolation in $\mathbb Z^d$. We write $rS = \{rs~:~s \in S\}$ for any subset $S \subseteq \mathbb{R}^d$ and $r\in \mathbb{R}$.

\begin{theorem}[Cox and Durrett \cite{CoxDurrett}]
If \eqref{eq:conditionmomentsLS} and \eqref{eq:conditionmomentsLS2} hold, then the first passage percolation model has a limit shape. That is, there exists a deterministic, convex, compact set $\mathcal B$ in $\mathbb{R}^d$, such that for each $\epsilon>0$,
\begin{equation*}\label{eq:limitshapeeq}
\mathbb{P}\left( (1-\epsilon)\mathcal B \subset \frac{B(t)}{t} \subset (1+\epsilon) \mathcal B \text{ for all } t \text{ large}\right) = 1.
\end{equation*}
Moreover, the limit shape $\mathcal B= \{ x \in \mathbb R^{d}: \mu(x) \leq 1 \}$ has a non-empty interior and is symmetric about the axes of $\mathbb R^{d}.$
\end{theorem}
\noindent Despite the importance of both objects, our knowledge on $\mu(e_{1})$ and $\mathcal B$ is almost non-existent. Finding a distribution where one can explicitly determine $\mu(e_{1})$  is considered a difficult open task, and so is deriving further properties of $\mathcal B$. The main purpose of this paper is to investigate such questions for $d$ large.

We assume that the passage times have finite mean, 
\begin{align}
\label{eq:dist122}
\E \tau_{e}=\int_{0}^{\infty} x dF(x) < \infty
\end{align}
and the existence of some constant $a \in [0,\infty]$ such that,
\begin{align}
\label{eq:dist2}
\bigg | \frac{\Pro(\tau_{e} \le x)}{x} - a \bigg | \le C\cdot |\log x|^{-1},
\end{align}
for some $C > 0$ in some interval $[0,\epsilon_{0}]$, $\epsilon_{0} >0$. Here, we understand $a=\infty$ as

\begin{equation}\label{eq:todayisFriday}
 \lim_{x \to 0} \frac{\Pro(\tau_{e} \le x)}{x} = \infty. 
\end{equation}

Our first main result is the asymptotic behavior of $\mu(e_{1})$ as a function of $d$.

\begin{theorem} \label{thm:Thisisthefirsttheorem} Assume \eqref{eq:dist122}  and that  \eqref{eq:dist2} holds for some $a \in [0,\infty]$. Then the time constant satisfies

\begin{equation}\label{eq:limitofthetimeconstant}
\lim_{d\to \infty}\frac{\mu(e_{1})d}{\log d} = \frac{1}{2a}.
\end{equation}

\end{theorem}

We now put the theorem above  into historical context. The behavior of the time constant $\mu(e_{1})$ as a function of $d$ was considered before by Kesten \cite[Section 8]{Aspects} and Dhar \cite{Dhar}. Their assumptions on the distribution of the passage times are special cases of ours. First, in \cite{Aspects}, under \eqref{eq:dist122}, \eqref{eq:dist2} and the additional assumptions that $a\in (0,\infty)$, $C=o(1)$ as $x\to 0$ and $\tau_{e}$ has a density around the origin, Kesten showed the existence of $\epsilon>0$ so that 
\begin{equation}\label{eq:SiisdrawingaHistogram}
\frac{\epsilon}{a} < \liminf_{d\to \infty} \frac{\mu(e_{1})d}{\log d}  \leq \limsup_{d\to \infty} \frac{\mu(e_{1})d}{\log d} \leq \frac{11}{a}.
\end{equation}

Second, in \cite{Dhar}, Dhar established \eqref{eq:limitofthetimeconstant} in the case of exponentially distributed passage times. Dhar's proof however cannot be adapted to any other distribution as, for instance, it heavily relies on the Markovian property of the ball $B(t)$. Thus, Theorem~\ref{thm:Thisisthefirsttheorem} says that the asymptotics obtained by Dhar are valid under rather general assumptions, that include those of Kesten.

\begin{rem}
Hypothesis \eqref{eq:dist122} is a natural condition on the behavior of the distribution of the passage times at $0$. It is satisfied by a large collection of examples; for instance, it includes all distributions that have a continuous density near the origin. The bound with $|\log x|^{-1}$ is a weaker condition than any polynomial bound around $0$.
\end{rem}
\begin{rem}
 If the distiribution has a mass at $0$ then \eqref{eq:todayisFriday} clearly holds. In this case, Theorem 6.1 in \cite{Aspects} and the fact that the critical probability $p_{c}(d)$ for bond percolation in $\mathbb Z^{d}$ decreases to $0$ as $d$ goes to infinity imply that the sequence $(\mu(e_{1}))_{d\geq 1}$ will be eventually constant equal to $0$ and, of course, \eqref{eq:limitofthetimeconstant} holds. 
\end{rem}

A word of comment is needed here. If, for some $\delta>0$, the support of the distribution of the passage times is included in $(\delta,\infty)$, then it is clear that \eqref{eq:limitofthetimeconstant} must hold. Indeed, in this case, $a =0$ and $\mu(e_{1}) \geq \delta$ for all $d$. It would be interesting to further study the behavior of $\mu(e_{1})$ in this situation. As we will see, this question seems to be related to the typical length (number of edges) of a geodesic and the behavior of $p_{c}(d)$ as a function of $d$.

Our second main result excludes the $d$-dimensional Euclidean ball $\mathsf B:=\{ x \in \mathbb R^{d}: \|x \|_{2} \leq \mu(e_{1})^{-1} \} $,  cube $\mathsf C:=\{ x \in \mathbb R^{d}: \|x\|_{\infty}\leq \mu(e_{1})^{-1} \}$  and  diamond  $\mathsf D:=\{ x \in \mathbb R^{d}:  \|x\|_{1}\leq \mu(e_{1})^{-1} \} $ as possible limit shapes. Note that due to convexity we always have  $\mathsf D \subseteq \mathcal B \subseteq  \mathsf C$. 

\begin{theorem} \label{thm:YellowChocobo} For any distribution satisfying \eqref{eq:conditionmomentsLS} and \eqref{eq:conditionmomentsLS2} for all $d \geq 2$, and \eqref{eq:dist122}, \eqref{eq:dist2} with $a\in (0,\infty)$, there exists $d_{0}  \geq 1$ such that for any $d \geq d_{0}$,
$$\mathsf D \subsetneq \mathcal B \subsetneq \mathsf C \text{ and } \mathcal B \neq \mathsf B.$$  

\end{theorem}

\begin{rem}
We prove Theorem \ref{thm:YellowChocobo} by showing that the intersection of $ \mathcal B$ with the line $\ell_{d}:=\{\lambda (1,\ldots,1):\lambda \in \mathbb R \}$ is strictly contained in $\mathsf B$ and strictly  contains $\mathsf D \cap \ell_{d}$. Using symmetry around $\ell_{d}$ the proof may further exclude other possible limit shapes. We delay the proof until Section 5.
\end{rem}

\begin{rem}
One of the main features of the Theorem above is that $d_{0}$ can be explicitly estimated for any given values of $a$ and $C$ in \eqref{eq:dist2}. In the case of an exponential random variable or a uniform random variable on some interval $[0,s]$, we show that $d_{0} = 269,000$ is sufficient  (but  certainly not optimal). We exclude the $d$-dimensional diamond for all $d\geq 110$. See Appendix.
\end{rem}

The fact that limit shape is not an Euclidean ball is expected to hold for all $d \geq 2$. Kesten provided the first results in this direction. He showed (see \cite[Remark 8.5]{Aspects}) that this is the case for the exponential distribution if $d \geq 10^{6}.$ 

\begin{rem}
In \cite{MR2753301}, Couronn\'{e}, Enriquez and Gerin also considered FPP with exponential distributed passage times. They provided a constructive way to find an upper bound of order $\log d/d $ for $\mu(e_{1})$. They also claimed that the limit shape is not an Euclidean ball if $d \geq 35$. However, the argument presented in \cite[Corollary 4]{MR2753301} seems unclear to us. Their claim is obtained using numerical results provided in Table $1$ of Dhar \cite{Dhar}. It is unclear whether the inequality $\mu(35) \leq 0.93 \frac{\log 2d}{2d}$  appearing in \cite{MR2753301} implies $\mu(d) \leq C(d) \frac{\log 2d}{2d}$ for some $C(d)$ that leads to the result for $d>35$. In view of \eqref{eq:limitofthetimeconstant}, the constant $C(d)$ must approach $1$ as $d \to \infty$, even if the numbers appearing in Table 1 in \cite{Dhar} are monotonically decreasing.
\end{rem}

The rest of the paper is organized as follows. In Section \ref{sec:ub}, we will sketch the proof of Theorem \ref{thm:Thisisthefirsttheorem}. The two following sections are devoted to prove the bounds $ {\limsup_{d\to \infty} \frac{\mu(e_{1})ad}{\log d} \leq 1/2}$ and $  \liminf_{d\to \infty} \frac{\mu(e_{1})ad}{\log d} \ge 1/2$, respectively. In Section \ref{sec:Kuppo}, we prove Theorem \ref{thm:YellowChocobo} by deriving a lower bound for the time constant in the diagonal direction. In the last section, we provide the quantitative bounds to control $d_{0}$.  Throughout the paper, we use $e_{1}, e_{2}, \ldots, e_{d}$ to denote the canonical base vectors of $\mathbb Z^{d}$.

\subsection*{Acknowledgments}
Both authors thank the American Institute of Mathematics for their support during the workshop ``First passage percolation and related models'', in August 2015, where this project was initiated. Si Tang thanks the Department of Mathematics at Northwestern University for the hospitality during her visits.

\section{Proof Strategy of Theorem \ref{thm:Thisisthefirsttheorem}}\label{sec:strategy}
The proof strategy is motivated by \cite{Aspects}. The reader will see that the upper bound 
\begin{equation}\label{SiiswearingRedtoday}
{\limsup_{d\to \infty} \frac{\mu(e_{1})ad}{\log d} \leq 1/2}
\end{equation}
contains the most technical part and the main new ideas of the paper. The lower bound 
$$\liminf_{d\to \infty} \frac{\mu(e_{1})ad}{\log d} \ge 1/2 $$
follows closely from \cite{Aspects} and is less intricate. In this section we explain how to derive \eqref{SiiswearingRedtoday} and the differences between our approach and Kesten's proof of \eqref{eq:SiisdrawingaHistogram}.

First, it is known (see for instance \cite[pp.246]{Aspects}) that
\[
\mu(e_{1}) \le \E \tilde s_{0,1}
\]
where 
\begin{align*}
\tilde s_{0, n} &\mathdef \inf \left\{
\begin{array}{c}
T(\gamma): \gamma \text{ is a path from }(0, 0, \ldots, 0)\text{ to some point in } H_{n}\\
\text{which, except for its {\em final} point, is contained in } [0, 1) \times \mathbb R^{d-1}
\end{array}
\right\}\\
H_{n}&\mathdef\{(x_{1}, \ldots, x_{d}) \in \mathbb Z^{d}: x_{1}=n\}.
\end{align*}
We will derive the upper bound for $\mu(e_{1})$ by bounding $\E \tilde s_{0,1}$ from above. The idea is to look for a path of length $n=n(d)$ to $H_{1}$ that has a very small passage time, and it is contained in a subspace $\mathcal H$ of dimension $p=p(d) < d$. To prove the existence of such a favorable path, the second moment method is a natural approach.  In \cite{Aspects},  Kesten  took $p=\lfloor d/2\rfloor$ and considered directed paths whose first $(n-1)$ steps go along the positive directions $+e_{2}, \ldots, +e_{p+1}$ and then, at the last step, take the $e_{1}$ direction to reach $H_{1}$. As a result, these paths are necessarily self-avoiding, which allows an estimate of the passage time using the sum of i.i.d. random variables. Since these paths only take the positive directions of $+e_{2}, \ldots, +e_{p+1}$, no more than $\lfloor d/2\rfloor^{(n-1)}$ paths were considered.\Br
Kesten's proof was then a trade-off between examining a large collection of  paths and being able to estimate $\tilde s_{0,1}$ using sum of i.i.d. random variables.  His strategy led to the upper bound in \eqref{eq:SiisdrawingaHistogram}. In order to get an optimal upper bound, we explore a subspace that is almost as large as the entire $\mathbb Z^{d}$ by choosing $p=d-o(d)$. Furthermore, we allow paths to go along any of the $2p$ possible directions, $\pm e_{2}, \ldots, \pm e_{p+1}$. Under such setting, we are able to examine nearly all of the paths leading to $H_{1}$ and obtain the optimal upper bound. The price to pay is that now some paths will be self-intersecting and thus we can not approximate their passage times by a sum of i.i.d.'s. Furthermore, the computation in the second moment method becomes elaborate. In the end of the day, the price is affordable as in high dimension, the majority of random walk paths are self-avoiding. The main estimation is done by carefully counting patterns of overlapping segments for a given pair of random walks in $\mathbb Z^{d}$. This main step is done is Section \ref{sec:kdosadkasokpojigorhughrurejioef}.

\section{Proof of \eqref{SiiswearingRedtoday} \label{sec:ub}}
\subsection{Setup} We are interested in the self-avoiding paths of length $n$ from 0 to $H_{1}$ whose first $(n-1)$ steps use directions $\pm e_{2}, \ldots, \pm e_{p+1}$ and the last step is $e_{1}$. Denote by $\mathcal P_{n}$ the set of all such paths.  For $\gamma \in \mathcal P_{n}$, we write it as
\[
\gamma = (S_{0}=0, S_{1}, S_{2}, \ldots, S_{n}),
\]
where $S_{k} \in \mathbb Z^{d}$ such that (i) $S_{i}\ne S_{j}$ whenever $i \ne j$, (ii) $S_{k}-S_{k-1} \in \{\pm e_{2}, \ldots, \pm e_{p+1}\}$ for $1\le k\le n-1$, and (iii) $S_{n}-S_{n-1}=e_{1}$.
Let $N_{n, x}$ be the number of paths $\gamma \in \mathcal P_{n}$ such that $T(\gamma) \le x$. We choose 
\begin{equation}
\label{eq:choicenxp}
n =\lfloor \log d \rfloor,\quad x = \frac{\log d}{2(1-\delta)a d},\quad p=d-\left \lfloor \frac{d}{\delta^{1+\eta} \log d}\right \rfloor
\end{equation}
for some $\delta, \eta>0$ fixed, but we will eventually send $\delta$ to 0. 
\subsection{First moment of $N_{n, x}$}
By definition, $\E N_{n, x}$ can be written as
\begin{align}
\notag
\E N_{n, x} &= \sum_{\gamma \in \mathcal P_{n}} \Pro(T(\gamma) \le x) =|\mathcal P_{n}|\cdot \Pro(S_{n} \le x)
\end{align}
We will need the following two lemmas to estimate $|\mathcal P_{n}|$ and  $\Pro(S_{n}\le x)$. The first one is about the number of self-avoiding walks, for which the estimate has been improved over the years \cite{Kesten:1964, Hammersley:1963, Madras:2013em} .
\begin{lem} \label{lem:sawsrw} Let $C_{n, d}$ denote the number of $d$-dimensional self-avoiding walks of length $n$. 
\begin{enumerate}[\normalfont (a)]
\item  $\xi_{d} \mathdef \lim_{n\to \infty} C_{n, d}^{1/n}$ exists and 
$
\xi_{d} \ge 2d-1-\log (2d-1)
$ for all $d \ge 1$.
\item 
As $d\to \infty$, $\xi_{d}$ has the following expansion
\[
\xi_{d} = 2d-1-\frac{1}{2d} - \frac{3}{(2d)^{2}} + O\left(\frac{1}{(2d)^{3}}\right).
\]
\end{enumerate}
\end{lem}
\begin{lem}  \label{lem:kesten8.8}Let $X_{1}, X_{2}, \ldots, $ be i.i.d. nonnegative random variables, satisfying \eqref{eq:dist2}.
Let $S_{n}\mathdef \sum_{i=1}^{n}X_{i}$ be the partial sum. Then, for all $n\ge 1$ and $0\le x \le \epsilon_{0}$, there is
\[
\frac{(ax)^{n}}{n!}(1-C|\log x|^{-1})^{n} \le \Pro(S_{n}\le x) \le \frac{(ax)^{n}}{n!}(1+C|\log x|^{-1})^{n}
\]
\end{lem}
\begin{proof} The result follows from \eqref{eq:dist2} and a similar calculation as in \cite[Lemma 8.8]{Aspects}.
\end{proof}

\noindent By Lemma \ref{lem:sawsrw}(a) and sub-additivity \cite[pp.9]{Madras:2013em},  we know that 

\[
[2p-1-\log(2p-1)]^{n-1} \le \xi_{p}^{n-1}\le |\mathcal P_{n}| \le (2p)^{n-1}.
\]
Also, Lemma \ref{lem:kesten8.8} and Stirling's formula imply
\[
\frac{(axe)^{n}(1-C|\log x|^{-1})^{n}}{n^{n}e\sqrt{ n}} \le \Pro(S_{n}\le x) \le \frac{(axe)^{n}(1+C|\log x|^{-1})^{n}}{n^{n}\sqrt{2\pi n}}.
\]
Putting them together, we get
\[
\frac{[2p-1-\log(2p-1)]^{n-1}(axe)^{n}(1-C|\log x|^{-1})^{n}}{n^{n}e\sqrt{ n}} \le \E N_{n,x} \le \frac{(2paxe)^{n}(1+C|\log x|^{-1})^{n}}{2pn^{n}\sqrt{2\pi n}}.
\]
We look at the left side. Firstly, note that $x \to 0$ as $d \to \infty$. In particular, 
\[
|\log x| = \log (2(1-\delta)a ) + \log d - \log \log d = \log d(1+o(1)).
\]
Hence, with $n = \lfloor \log d\rfloor$ and for all $d$ large, we have
\begin{align*}
(1-C|\log x|^{-1})^{n} = e^{n\log (1-C|\log x|^{-1})}=e^{-Cn|\log x|^{-1} +O(n |\log x|^{-2}) } > e^{-2C}.
\end{align*}
This gives a lower bound for $\E N_{n, x}$ for all $d$ sufficiently large:
\begin{align}
\notag
\E N_{n,x} & \ge \frac{e^{-2C}\left(\frac{2p-1-\log(2p-1)}{2d(1-\delta)}\right)^{\log d}}{\frac{2p-1-\log(2p-1)}{d}e\sqrt{\log d}} \ge \frac{e^{-2C}d^{\log \frac{2p-1-\log(2p-1)}{2d(1-\delta)} }}{2e\sqrt{ \log d}} \\
\label{eq:1stmom}
&\ge \frac{ e^{-2C}d^{\log \frac{1 - \frac{1}{d}\left\lfloor \frac{d}{\delta^{1+\eta}\log d}\right \rfloor-\frac{1 + \log(2d)}{2d}}{1-\delta}}}{2e \sqrt{\log d}}.
\end{align}
For $\delta \in (0, 1)$ fixed and $d$ large, we will have $\delta > \frac{1}{d}\left\lfloor \frac{d}{\delta^{1+\eta}\log d}\right \rfloor + \frac{1+\log(2d)}{2d} $, which implies $\frac{1 -  \frac{1}{d}\left\lfloor \frac{d}{\delta^{1+\eta}\log d}\right \rfloor-\frac{1+\log (2d)}{2d}}{1-\delta} > 1$. Hence, $\E N_{n,x} \to \infty$ as $d\to \infty$. 
\subsection{Proof of the Upper Bound}
By definition, the second moment of $N_{n, x}$ can be written as
\begin{align*}
\E N_{n, x}^{2} &= \sum_{\gamma, \gamma'\in \mathcal P_{n}}  \Pro(T(\gamma) \le x, T(\gamma') \le x).
\end{align*}
Suppose for now that we are able to show that for some $0<A < \infty$, there is 
\begin{equation}
\label{eq:momineq}
\E N_{n, x}^{2} \le A (\E N_{n, x})^{2} 
\end{equation}
for all $d$ large.
By Cauchy-Schwarz inequality, we know
\[
\Pro(N_{n, x} \ge 1) \ge \frac{(\E N_{n, x})^{2}}{\E N_{n, x}^{2}} \ge \frac{1}{A} > 0.
\]
This means with positive probability, we can find a path $\gamma \in \mathcal P_{n}$ from $0$ to $H_{1}$ such that $T(\gamma) < x$. The proof of \eqref{eq:momineq} will be given in Section 2.5.\Br
Let $\mathcal H$ be the subspace spanned by $\pm e_{2}, \ldots, \pm e_{p+1}$. 
Now we focus on the coordinates $e_{p+2}, e_{p+3}, \ldots, e_{d}$. For $ p+2 \le j \le d$, let $E_{j}$ be the event  that there exists a path from $e_{j}$ to $H_{1}$ such that, except for its final point, is contained in $[0, 1) \times \mathbb R^{d-1} \cap (\mathcal H + e_{j})$ and has $T(\gamma) \le x$.  By translation invariance and \eqref{eq:momineq}, we have 
\[
\Pro(E_{j}) \ge 1/A > 0.
\] 
Choose $y=\frac{\delta\log d}{ad}$. For $ p+2 \le j \le d$, let $F_{j}$ be the event $\{\tau_{e_{j}} \le y\} \cap E_{j}$. By Lemma \ref{lem:kesten8.8},
\[
\Pro(F_{j}) = \Pro(\tau_{e_{j}} \le y) \Pro(E_{j}) \ge \frac{1}{A}\Pro(\tau_{e_{j}} < y) \ge \frac{1}{A}ay(1-C|\log y|^{-1}) \ge \frac{\delta \log d}{Ad} (1+o(1)).
\]
Furthermore, these $F_{j}$'s are independent, and if any of the events $F_{j}$ happens, we will have
\[
\tilde s_{0,1} \le y + x = \frac{\log d}{2ad}\left (2\delta + \frac{1}{1-\delta}\right).
\]
Therefore,
\begin{align}
\notag
\E\tilde s_{0,1} &\le (y+x) \E \mathbf 1_{\bigcup_{j=p+2}^{d}F_{j}} + \E \tau_{e_{1}} \E \mathbf 1_{\bigcap_{j=p+2}^{d}F^{c}_{j}} \\
\label{eq:finalcal}
&\le \frac{\log d}{2ad}\left (2\delta + \frac{1}{1-\delta}\right) + \left(1-\frac{\delta \log d}{Ad} (1+o(1))\right)^{d-p-1}\E \tau_{e_{1}}.
\end{align}
Notice that, for all $d$ sufficiently large,
\[
\left(1-\frac{\delta \log d}{Ad} (1+o(1))\right)^{d-p-1} \le 2\left(1-\frac{\delta \log d}{Ad} (1+o(1))\right)^{\left \lfloor \frac{d}{\delta^{1+\eta}\log d}\right \rfloor} \le 4e^{-\frac{1}{A\delta^{\eta}}}.
\]
Thus, for any $\eta>0$, the second term in \eqref{eq:finalcal} vanishes as $\delta \to 0$.  This gives us $\E \tilde s_{0,1} \le \frac{\log d}{ 2ad}$ as desired.
\subsection{Second moment of $N_{n,x}$} \label{sec:kdosadkasokpojigorhughrurejioef}

We are going to prove \eqref{eq:momineq} in this section. We first rewrite the second moment according to the number $l\le n$ of overlapping edges between $\gamma$ and $\gamma'$:
\begin{align}
\label{eq:2mom1}
\E N_{n, x}^{2} &= \sum_{l=0}^{n} \sum_{\gamma, \gamma' \in \mathcal P_{n}} \Pro(T(\gamma)\le x, T(\gamma')<x)\mathbf 1_{\{|\gamma \cap \gamma'|=l\}}.
\end{align}
Note that since we only consider $\gamma, \gamma'\in \mathcal P_{n}$, which are self-avoiding,  the condition $\{|\gamma \cap \gamma'|=l\}$ is defined with no ambiguity to be the number of edges in $\gamma$ that also appear in $\gamma'$ (or vice versa). In what follows, we always write 
\begin{equation}
\label{eq:gammas}
\gamma = (S_{0}=0, S_{1}, \ldots, S_{n-1}, S_{n}),\quad
\gamma' = (S_{0}'=0, S_{1}', \ldots, S_{n-1}', S_{n}').
\end{equation} 
When $l=n$, due to the fact that they both start from the origin and are self-avoiding, we know $\gamma = \gamma'$. In this case we have
\[
\sum_{\gamma, \gamma' \in \mathcal P_{n}} \Pro(T(\gamma)\le x, T(\gamma')<x)\mathbf 1_{\{|\gamma \cap \gamma'|=n\}} = \Pro(T(\gamma)\le n)\cdot |\mathcal P_{n}| = \E N_{n,x}. 
\]
When $l=0$, $\gamma$ and $\gamma'$ do not share any edges.
\[
\sum_{\gamma, \gamma' \in \mathcal P_{n}} \Pro(T(\gamma)\le x, T(\gamma')<x)\mathbf 1_{\{|\gamma \cap \gamma'|=0\}} \le \Pro(S_{n} \le x)^{2}\cdot |\mathcal P_{n}| ^{2} = (\E N_{n,x})^{2}.
\]
For other $1 \le l \le n-1$, we can write,
\begin{align*}
\Pro( T(\gamma) \le x, T(\gamma') \le x)\mathbf 1_{\{|\gamma \cap \gamma'|=l\}}  &\le \Pro(T(\gamma \setminus (\gamma\cap \gamma')) \le x, T(\gamma') \le x) \mathbf 1_{\{|\gamma \cap \gamma'|=l\}}  \\
&\le \Pro(S_{n-l}\le x) \Pro(S'_{n} \le x) \mathbf 1_{\{|\gamma \cap \gamma'|=l\}}.
\end{align*}
Lemma \ref{lem:kesten8.8} implies that for all $1\le l \le n-1$ and $d$ large,
\begin{align}
\label{eq:Aconst1}
\frac{\Pro(S_{n-l} \le x)}{\Pro(S_{n}\le x)} &\le \frac{\frac{(ax)^{n-l}}{(n-l)!} (1+C|\log x|^{-1})^{n-l}}{\frac{(ax)^{n}}{n!} (1-C|\log x|^{-1})^{n}} \le \left(\frac{n}{ax}\right)^{l}\left(1+\frac{2C|\log x|^{-1}}{1-C|\log x|^{-1}}\right)^{n}\\
\notag
& \le \left(\frac{n}{ax}\right)^{l}e^{n \log \left[1+\frac{2C}{|\log x|-C}\right]}\le \left(\frac{n}{ax}\right)^{l}e^{\frac{2nC}{|\log x|-C}} = \left(\frac{n}{ax}\right)^{l}e^{f_{a,C}(\delta, d)},
\end{align}
where we denote 
\begin{equation}
\label{eq:Aconst1}
f_{a,C}(\delta, d) \mathdef \frac{2nC}{|\log x|-C} = \frac{2 C \log d }{\log[2(1-\delta)ad] - \log\log d-C}.
\end{equation}
Note, for fixed $a, C, \delta > 0$, $f_{a,C}(\delta, d) \to 2C$ as $d\to \infty$. 
Hence, for sufficiently large $d$, $f_{a,C}(\delta, d) < 3C$, and \eqref{eq:2mom1} becomes
\begin{align*}
\E N_{n,x}^{2} &\le \E N_{n,x} + (\E N_{n, x})^{2}+e^{3C}\Pro^{2}(S_{n} \le x) \sum_{l=1}^{n-1} \sum_{\gamma, \gamma' \in \mathcal P_{n}} \left(\frac{n}{ax}\right)^{l} \mathbf 1_{\{|\gamma \cap \gamma'|=l\}}\\
&\le  \E N_{n,x} + (\E N_{n, x})^{2}+e^{3C}\Pro^{2}(S_{n} \le x) \sum_{l=1}^{n-1}  \left(\frac{n}{ax}\right)^{l} \#\{(\gamma, \gamma'): \gamma, \gamma' \in \mathcal P_{n}, |\gamma \cap \gamma'|=l\}\\
&= \E N_{n,x} + (\E N_{n, x})^{2}+ e^{3C} [(2p)^{n-1}\Pro(S_{n} \le x) ]^{2}\sum_{l=1}^{n-1}  \left(\frac{n}{ax}\right)^{l} \frac{\# \{(\gamma, \gamma'): \gamma, \gamma' \in \mathcal P_{n}, |\gamma \cap \gamma'|=l\}}{(2p)^{2(n-1)}}.
\end{align*}
Note that $\E N_{n, x} \to \infty$ and the front factor of the second term satisfies, for $d$ large,
\begin{align*}
1\le \frac{(2p)^{n-1}\Pro(S_{n} \le x)}{\E N_{n,x}} &\le \left(1+\frac{1+\log(2p-1)}{2p-1-\log(2p-1)}\right)^{n-1}\le e^{g_{\eta}(\delta, d)},
\end{align*}
where we have denoted the exponent by 
\begin{align}
\notag
g_{\eta}(\delta, d) &= \frac{(n-1)(\log(2p-1) + 1)}{2p-1-\log(2p-1)} \\
\label{eq:Aconst2}
&\le \frac{(\log d)^{2}+ \log2\log d }{2d (1-1/(\delta^{1+\eta}\log d)) -1 - \log 2d}.
\end{align}
For $\delta, \eta>0$ fixed, and $d$ sufficiently large, $g_{\eta}(\delta, d) \to 0$ as $d \to \infty$. Hence, we can choose $d$ sufficiently large, such that $e^{g_{\eta}(\delta, d)} < 2$, which yields 
\[
\E N_{n,x}^{2} \le [\E N_{n,x}]^{2}\left\{ 1 + 4e^{3C}  \sum_{l=1}^{n-1}  \left(\frac{n}{ax}\right)^{l}\frac{\# \{(\gamma, \gamma'): \gamma, \gamma' \in \mathcal P_{n}, |\gamma \cap \gamma'|=l\}}{(2p)^{2(n-1)}} + o(1)\right\}. 
\]
To proceed, we need the following proposition. 
\begin{prop}\label{prop:mainprop}For each $1 \le l\le n-1$ fixed and $p$ sufficiently large, 
\begin{equation}
\label{eq:lintersect}
\frac{\# \{(\gamma, \gamma'): \gamma, \gamma' \in \mathcal P_{n}, |\gamma \cap \gamma'|=l\}}{(2p)^{2(n-1)}} \le \left(\frac{1}{2p}\right)^{l} (1+o(p^{-1/2})).
\end{equation}
\end{prop}
\noindent Given Proposition \ref{prop:mainprop}, we have
\begin{align*}
\E N_{n,x}^{2} &\le  [\E N_{n,x}]^{2} \left [2+8e^{3C} \sum_{l=0}^{n-1}\left(\frac{n}{2pax}\right)^{l}\right]\\
&\le [\E N_{n, x}]^{2}\left [ 2+ 8e^{8C} \sum_{l=0}^{\infty} \left(\frac{1-\delta}{1-1/(\delta^{1+\eta}\log d) }\right)^{l}\right].
\end{align*}
Again, for $\delta > 0$ fixed, and $d$ sufficiently large, we have $\frac{1-\delta}{1-1/(\delta^{1+\eta}\log d)} < 1$, hence the summation above converges, which gives \eqref{eq:momineq}. The proof for the upper bound is then complete once we show \eqref{eq:lintersect}.\Br
When $\gamma, \gamma' \in \mathcal P_{n}$ and $|\gamma\cap \gamma'|=l$, there are two cases: (i) if $S_{n-1}\ne S'_{n-1}$, it is necessary that all $l$ overlapping edges occur in the first $(n-2)$ steps since both $\gamma$ and $\gamma'$ take the $e_{1}$ direction at the last step; (ii) if $S_{n-1} = S'_{n-1}$, then $\gamma$ and $\gamma'$ share the last edge so that there are at most $(l-1)$ overlapping edges in their first $(n-1)$ steps. Observe that the denominator of \eqref{eq:lintersect} is just the number of all pairs of paths in $\mathbb Z^{p}$ of length $(n-1)$, starting from the origin.  Hence, if we put uniform measure on all pairs of simple random walk paths $(\tilde \gamma, \tilde \gamma')$ in $\mathbb Z^{p}$ starting from the origin and of length $(n-1)$ , the left side of \eqref{eq:lintersect} is bounded by
\begin{align}
\label{eq:bubblecase1}
\Pro\bigg(\{\tilde \gamma, \tilde \gamma' \text{ are self-avoiding}\} \cap \{|\tilde \gamma \cap \tilde \gamma'| = l \}\bigg)\qquad \qquad \cdots \quad &\text{Case (i)}\\
\label{eq:bubblecase2}
+ \ \Pro\bigg(\{\tilde \gamma,\tilde  \gamma' \text{ are self-avoiding}\} \cap \{|\tilde \gamma \cap \tilde \gamma'| = l-1 \} \cap \{S_{n-1}=S_{n-1}'\}\bigg)\quad \cdots\quad &\text{Case (ii)}
\end{align}
Here $\tilde \gamma$ and $\tilde  \gamma'$ can be thought as the first $(n-1)$ steps of $\gamma$ and $\gamma'$, respectively, and we abuse the notation by writing 
\begin{equation}
\label{eq:gammas2}
\tilde \gamma = (S_{0}=0, S_{1}, \ldots, S_{n-1}),\quad
\tilde \gamma' = (S_{0}'=0, S_{1}', \ldots, S_{n-1}').
\end{equation} 
We prove Case (i) and Case (ii) in Lemma \ref{lem:main} (i) and (ii), respectively. The proofs for both cases are very similar, based on counting the ``bubbles'' of two intersecting simple random walk paths. We explain in full detail the construction in Case (i), whereas for Case (ii), we just point out the difference. \Br
Let $\gamma$ and $\gamma'$ be two paths sampled uniformly and independently from all simple random walk paths in $\mathbb Z^{p}$, starting from the origin, and of length $n \le \lfloor 10\log p\rfloor $. Note that the constant $10$ takes into account that we are actually interested in paths of length $\lfloor \log d \rfloor -1$ and $p \approx d$. Such differences are negligible when $d$ is large. 
\begin{lem} \label{lem:main}
For all $2\le n\le \lfloor 10\log p \rfloor$,  $1\le l \le n-1$, and sufficiently large $p$, we have
\begin{enumerate}[\normalfont (i)]
\item 
$\Pro\big(\{\gamma, \gamma' \text{ are self-avoiding}\} \cap \{|\gamma \cap \gamma'| = l \}\big) \le \left(\frac{1}{2p}\right)^{l} (1+o(p^{-1/2}))$,
\item 
$
\Pro\big(\{\gamma, \gamma' \text{ are self-avoiding}\} \cap \{|\gamma \cap \gamma'| = l-1 \} \cap \{S_{n}=S'_{n}\}\big) \le \left(\frac{1}{2p}\right)^{l} o(p^{-1/2}).
$
\end{enumerate}
\end{lem}
Case (i) and (ii) follow from Lemma \ref{lem:main}, when we replace $n$ by $(n-1)$ and $\gamma, \gamma'$ by $\tilde \gamma, \tilde \gamma'$, respectively.
\begin{proof}[Proof of Lemma \ref{lem:main} \normalfont (i)] $ $\\ 
\vspace{0.1cm}

For each $1\le K\le l$, let $B_{K}$ denote the event that the $l$ overlapping edges are clustered in $K$ consecutive pieces.
Let $n_{1}, \ldots, n_{K}$ be the lengths of these segments, $m_{1}, \ldots, m_{K}$ (resp. $m_{1}', \ldots, m_{K}'$) be the indices of their starting points in $\gamma$ (resp. in $\gamma'$). For example, in the left part of Figure \ref{fig:fig111}, we have $l=5$ and $K=2$. For convenience, we also denote by $A_{l}$ the event $|\gamma \cap \gamma|=l$ and $G_{\gamma}$ (resp. $G_{\gamma'}$) the event that $\gamma$ (resp. $\gamma'$) is self-avoiding. The original probability is just
\[
\Pro(G_{\gamma}\cap G_{\gamma'} \cap A_{l} ) = \sum_{K=1}^{l}\Pro(G_{\gamma}\cap G_{\gamma'}\cap A_{l}\cap B_{K}).
\]
Note that, on the event $G_{\gamma}\cap G_{\gamma'}$, the number of segments $K$ must be the same in both $\gamma$ and $\gamma'$, i.e., the situations in the middle and on the right of Figure \ref{fig:fig111} cannot happen. 
 
\begin{figure}[ht]
\scalebox{0.9}{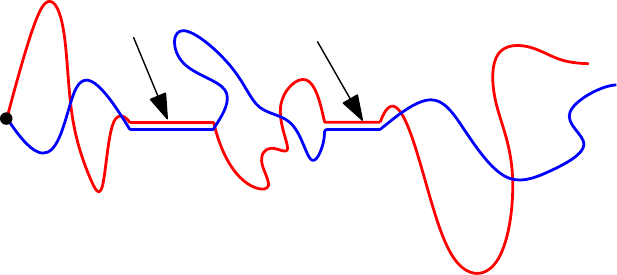}
\scalebox{1.3}{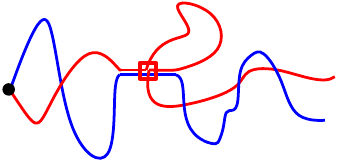}
\scalebox{0.6}{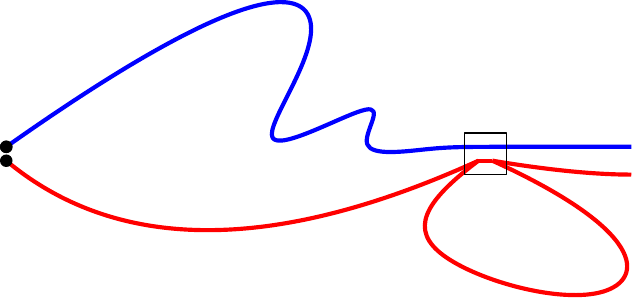}
\caption{An illustration of the definition of $K$ and $n_{1}, \ldots, n_{K}$. Parallel segments in the figure represent overlapping edges. On the left, two overlapping segments with the same orientation. In the middle, two consecutive edges in $\gamma$ overlap with two distant edges in $\gamma'$. On the right, two edges in $\gamma$ overlap with the same edge in $\gamma'$. The last two situations cannot occur in the event $G_{\gamma}\cap G_{\gamma'}$. }
\label{fig:fig111}
\end{figure}   

\vspace{0.1cm}
When $K=1$, the $l$ overlapping edges are clustered in one segment. There are two possible alignments for the overlapping segment, either along the same or along different directions (see Figure \ref{fig:easy}). In both cases, nonoverlapping pieces of $\gamma$ and $\gamma'$ form a ``bubble''-like shape.

The first situation happens with probability no more than
\begin{align*}
\sum_{m_{1}=0}^{n-l}\sum_{m_{1}'=0}^{n-l}\Pro(S_{m_{1}}=S_{m_{1}'}' ) \left(\frac{1}{2p}\right)^{l}
&=\left(\frac{1}{2p}\right)^{l} \left[ 1 + \sum_{m_{1}=0}^{n-l}\sum_{m_{1}'=0}^{n-l}\Pro(S_{m_{1}}=S_{m_{1}'}') \mathbf 1_{\{m_{1}+m_{1}' >0\}}\right]\\
&=\left(\frac{1}{2p}\right)^{l} \left[ 1 + \sum_{m_{1}=1}^{n-l}\sum_{m_{1}'=1}^{n-l}    \Pro(S_{m_{1}+m_{1}'}=0) \right]\\
&\le\left(\frac{1}{2p}\right)^{l} \left[ 1 + \frac{(n-l)^{2}}{2p}  \right].
\end{align*}
We have used the fact that if $\{S_{m}\}_{m\ge 0}$ is a simple random walk on $\mathbb Z^{p}$, then for all $m \ge 1$ and $t \in \mathbb Z^{p}$,
\begin{equation}
\label{eq:fact1}
\Pro(S_{m}=t) \le \sum_{s\in \mathbb Z^{p}}\Pro(S_{m-1}=s)\Pro(S_{1}=t-s) \le \frac{1}{2p} \sum_{s\in \mathbb Z^{p}}\Pro(S_{m-1}=s)\le \frac{1}{2p}.
\end{equation}
We write all terms for the second situation (Figure \ref{fig:easy}, right), as this will give us the spirit of the general $K > 1$ case. The probability is no more than
\begin{align*}
\sum_{m_{1}, m_{1}'=0}^{n-l-1}\sum_{t, t', s, s'\in \mathbb Z^{p}}  \Pro 
\left(
\begin{array}{c}S_{m_{1}}=t, S_{l+m_{1}}-S_{m_{1}}=s,\\
S_{m_{1}'}'=t', S_{l+m_{1}'}'-S_{m_{1}'}'=s', \\
S_{m_{1}'+j}'=S_{m_{1}+l-j}, 0\le j\le l
\end{array}\right)\mathbf 1_{\{m_{1}+m_{1}' > 0\}}.
\end{align*}

\begin{figure}[ht]
\scalebox{0.6}{ 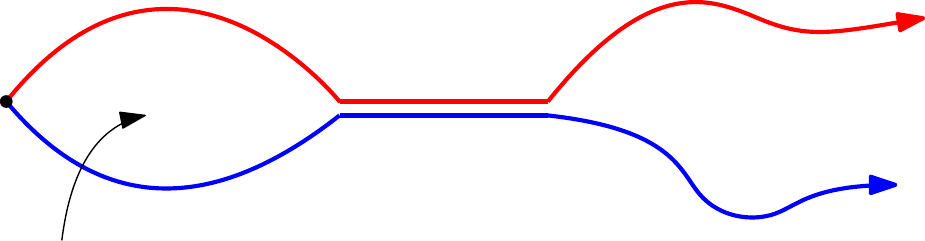}
\scalebox{0.5} {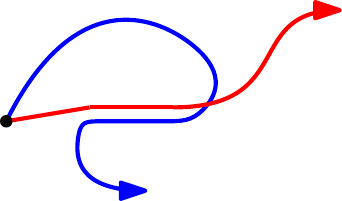}
 \centering
 \def \svgwidth{3000pt}
 \caption{Case $K=1$. The two paths $\gamma$ and $\gamma'$ intersect in a single segment and produce one bubble. They overlap along the same direction (left) or along opposite directions (right).  Parallel segments in the figure represent overlapping edges.}
 \label{fig:easy}
\end{figure}
Note that in this case, we can not have $m_{1}=m_{1}'=0$, because otherwise the overlapping segments would have aligned along the opposite direction. Conditioning on the events $\{ S_{m_{1}'}'=t', S_{l+m_{1}'}'-S_{m_{1}'}'=s'\} $ and $\{S_{m_{1}'+j}'=S_{m_{1}+l-j}, 0\le j\le l\}$, there is only one possibility for the choices of $t$ and $s$, i.e., $s=-s',\ t=t'-s'$. Moreover, the event $S_{l+m_{1}} - S_{m_{1}} =-s'$ is implied by the event $S'_{m_{1}'+j}=S_{m_{1}+l-j}$, $0\le j\le l$. Thus, using the independence between $\gamma$ and $\gamma'$, we can write
\begin{align*}
\left(\frac{1}{2p}\right)^{l}&  \sum_{m_{1}, m_{1}'=0}^{n-l}\sum_{t', s'\in \mathbb Z^{p}}  \Pro 
\left(S_{m_{1}}=t'-s'\right)\Pro(S'_{m_{1}'}=t')\Pro(S'_{m_{1}'+l}-S'_{m_{1}'}=s')\mathbf 1_{m_{1}+m_{1}' > 0}\\
&\le \left(\frac{1}{2p}\right)^{l}  \sum_{m_{1}, m_{1}'=0}^{n-l}\sum_{t', s'\in \mathbb Z^{p}} 
\frac{1}{2p}\cdot \Pro(S'_{m_{1}'}=t')\Pro(S'_{m_{1}'+l}-S'_{m_{1}'}=s')\mathbf 1_{m_{1}+m_{1}' > 0}\\
&\le \left(\frac{1}{2p}\right)^{l}  \sum_{m_{1}, m_{1}'=0}^{n-l}\frac{1}{2p} \le \left(\frac{1}{2p}\right)^{l} \frac{(n-l)^{2}}{2p}.
\end{align*}
Combining the two situations and using $n=C_{0}\log p$, we conclude
\[
\Pro(G_{\gamma}\cap G_{\gamma'}\cap A_{l}\cap B_{1}) \le \left(\frac{1}{2p}\right)^{l}(1+o(p^{-1/2})).
\]
From the $K=1$ case, we have the following observations, and the last observation is the most important for the general $1< K \le l$ case. 
\begin{enumerate}[\normalfont a.]
\item The overlapping edges give the factor $(1/2p)^{l}$.
\item There are either $(K-1)$ or $K$ bubble segments (i.e., those that do not overlap with the other path) in $\gamma$ and $\gamma'$, depending on whether $m_{1}$ and $m_{1}'$ are both zero or not. The total number of bubble and overlapping segments is either $2K-1$ or $2K$. 
\item Conditioning on the values of all $2K-1$ or $2K$ segments of $\gamma'$, there is at most one solution for the values on the segments of $\gamma$. The event that ``those $K$ overlapping segments in $\gamma$ take particular values'' is absorbed in the event that ``each edge in these segments overlaps with $\gamma$'', whereas the event ``the bubble segments take particular values'' occurs with probability no more than $\left(\frac{1}{2p}\right)^{K-1}$, due to \eqref{eq:fact1}.
\end{enumerate}
We now proceed to the general $1 < K \le l$ case. For each $K$ fixed, we first compute the number of ways to divide the $l$ overlapping edges into $K$ (nonempty) groups of sizes $n_{1}, n_{2}, \ldots, n_{K}$, which is no more than $\binom{l-1}{K-1}\le l^{K-1}$.  Next, we determine the positions of these $K$ overlapping segments in $\gamma$: there are $\binom{n}{K} $ ways to choose the starting points $m_{1} < m_{2} < \cdots < m_{K}$ of these segments, and once we have the starting points, we have $K!$ ways to associate a running length $n_{j_{i}}$ to a starting point $m_{i}$. This is over-counting, because if, say, $m_{2}-m_{1} < n_{3}$, then the overlapping segment starting at $m_{1}$ can not be longer than $n_{3}$.  We do the same for $\gamma'$. Once we have the locations of overlapping segments in $\gamma$ and $\gamma'$, we know exactly which segment in $\gamma$ overlap with which segments in $\gamma'$. There is an additional factor $2^{K}$, which counts for the two possible directions of alignment for each overlapping segment. These together give us a combinatorial number no larger than
\[
l^{K-1}\cdot \left[\binom{n}{K} \cdot K!\right ]^{2} 2^{K} \le \frac{(2ln^{2})^{K}}{l}, 
\]
which is an upper bound of all possible overlapping patterns (one of these is illustrated in Figure 3). Each pattern occurs with probability no more than
\[
\left(\frac{1}{2p}\right)^{l} \left(\frac{1}{2p}\right)^{K-1},
\]
and hence for all $2\le K \le l$, 
\[
\Pro(G_{\gamma}\cap G_{\gamma'}\cap A_{l}\cap B_{K}) \le \left(\frac{1}{2p}\right)^{l} \left[\frac{2l^{\frac{K-1}{K}}n^{2}}{(2p)^{\frac{K-1}{K}}}\right]^{K} \le  \left(\frac{1}{2p}\right)^{l} \left(\frac{ln^{2}}{p}\right)^{K}\frac{2p}{l} .
\]
\begin{figure}[ht]
\scalebox{1.2}{ 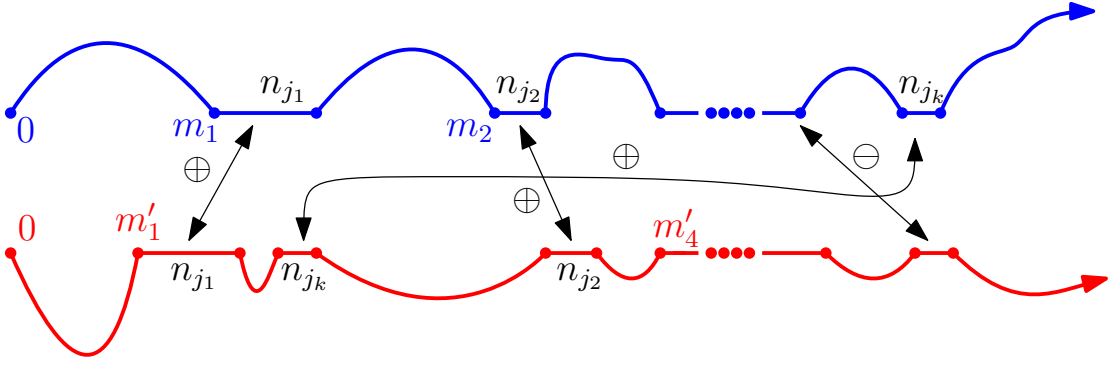}
\caption{Overlapping paths $\gamma$ and $\gamma'$ and their matching:  we split the $l$ overlapping edges into groups of sizes $n_{1}, n_{2}, \dots$, starting at points $m_{1}, m_{2}, \ldots$. Each overlapping piece can meet in a positive or negative orientation.}
\end{figure}
\\
Summing over $K$ and using the fact that $K \le l \le n \le C_{0}\log p$, we obtain
\[
\frac{2p}{l}\sum_{K= 2}^{l} \left(\frac{1}{2p}\right)^{l} \left(\frac{ln^{2}}{p}\right)^{K} = \left(\frac{1}{2p}\right)^{l}\frac{2ln^{4}}{p}\sum_{K= 0}^{l-2}\left(\frac{ln^{2}}{p}\right)^{K}  = \left(\frac{1}{2p}\right)^{l}\frac{2ln^{4}}{p} \frac{1}{1-\frac{ln^{2}}{p}}= \left(\frac{1}{2p}\right)^{l}o(p^{-1/2}),
\]
which finishes the proof of Lemma \ref{lem:main} (i).
\end{proof}
\begin{proof}[Proof of Lemma \ref{lem:main} \normalfont (ii)] When $K \ge 3$, there are at least $(K-1) \ge 2$ bubble segments in both $\gamma$ and $\gamma'$. In this case, we can simply ignore the last event $\{S_{n}=S'_{n}\}$ when calculating the probability. Following the same strategy as in Case (ii), one can easily get
\begin{align*}
&\left(\frac{1}{2p}\right)^{l-1}\sum_{K = 3}^{l} \frac{(2ln^{2})^{K}}{l\cdot [2p]^{K-1}} = \left(\frac{1}{2p}\right)^{l}\sum_{K=1}^{l-2} \frac{(2ln^{2})^{K+2}}{l\cdot [2p]^{K}} =  \frac{4l^{2}n^{6}}{p} \left(\frac{1}{2p}\right)^{l}\sum_{K=0}^{l-3} \left(\frac{ln^{2}}{p}\right)^{K}\\
&\qquad \le \left(\frac{1}{2p}\right)^{l}  \frac{4l^{2}n^{6}}{p} \frac{1}{1-\frac{ln^{2}}{p}} = \left(\frac{1}{2p}\right)^{l} o(p^{-1/2}).
\end{align*}
When $K=1$ and if $m_{1}=m_{1}'= 0$ (i.e., $\gamma$ and $\gamma'$ overlap at the first $l-1$ edges) or $m_{1}=m_{1}'=(n-l+1)$ (i.e., $\gamma$ at $\gamma'$ overlap at the last $l-1$ edges), there is only bubble segment in both $\gamma$ and $\gamma'$. Also, on the event that they are both self-avoiding and $S_{0}=S_{0}', S_{n}=S_{n}'$, the overlapping edges must align in the same direction.  This case is illustrated in Figure \ref{fi:fig4}.

\begin{figure}[h]
\scalebox{1.0}{ 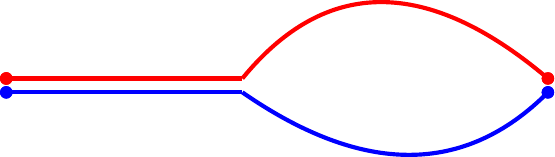}
\scalebox{1.0} {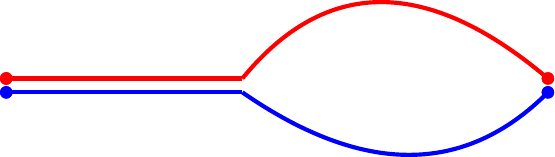}
\caption{Representation of the case $K=1$ with an aligned overlap. Only one bubble happens either at the end or the beginning of the paths.}
\label{fi:fig4}
\end{figure}
\begin{figure}\begin{centering}
\scalebox{0.9}{ 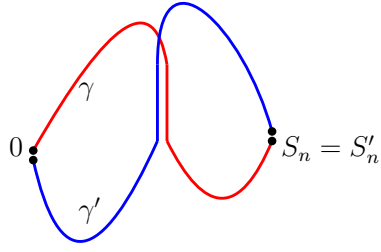}
\caption{Representation of the case $K=1$ with a negative alignment of the overlapping segment. In this case two bubbles must be created. Parallel segments represent identical edges.}
\label{fg:fi5}
\end{centering}
\end{figure}

In either situation, the probability of seeing such a ``bubble'' is no more than
\begin{equation}
\label{eq:fact2}
\Pro(\tilde S_{2} \ne 0, \tilde S_{m}=0) \le \frac{(m-2)^{2}}{(2p)^{2}},
\end{equation}
where $\tilde S_{m}$ denotes a simple random walk on $\mathbb Z^{p}$ of length $m \ge 3$. The  reason is that, conditioning on the event $\{\tilde S_{2}= x\}$ where $x\ne 0$, two of the next $(m-2)$ steps must go along the reverse directions of the coordinates used in the first two steps in order to return to the origin at the $m$-th step (note $m$ must be even). This happens with probability no more than $\frac{(m-2)^{2}}{(2p)^{2}}$. \Br
For all other values of $m_{1}$ and $m_{1}'$, there are at least two bubble segments (e.g., Figure \ref{fg:fi5}), which we can easily estimate using \eqref{eq:fact1}. We use $1=K \le l \le n \le C_{0}\log p$ again and compute
\[
2\cdot \left(\frac{1}{2p}\right)^{l-1} \frac{(2n-2)^{2}}{(2p)^{2}} + \left(\frac{1}{2p}\right)^{l-1} \frac{(2n^{2})^{1}}{(2p)^{2}} = \left(\frac{1}{2p}\right)^{l}o(p^{-1/2})
\]
For $K=2$, if $\gamma$ or $\gamma'$ has 2 bubble segments, we can use the strategy in Case (i) and \eqref{eq:fact1}. If there is only 1 bubble segments, then the two overlapping segments must attach to $S_{0}=S_{0}'=0$ and $S_{n}=S_{n}'$. Since $\gamma$ and $\gamma'$ are self-avoiding, this situation is quite similar to the fusion of bubbles in Figure 4 (a) and (b). The bubble in the middle can be easily estimated using \eqref{eq:fact2}.  We leave the details to the readers. 
\end{proof}

\section{Proof of the Lower Bound}\label{sec:lb}
In this section we establish the desired lower bound. We assume the existence of a constant $a \in [0,\infty)$ so that
\begin{align}
\label{eq:distLB}
\lim_{x\downarrow 0} \frac{\Pro(\tau_{e} \le x)}{x} = a.
\end{align}
Note that this condition is weaker than \eqref{eq:dist2}.

\begin{prop}\label{prop:LB} Assume $\mathbb E \tau_{e} <\infty$ and  \eqref{eq:distLB}. Then,
$$
\displaystyle  \liminf_{d\to \infty} \frac{\mu(e_{1})ad}{\log d}  \ge \frac{1}{2}.$$
\end{prop}
\noindent For the proof of Proposition \ref{prop:LB}, we will need a few lemmas. Let $b_{n} = T(0,H_{n})$ be the passage time from the origin to the hyperplane $H_{n}$. 
\begin{lem}\label{lem:basiclemmaBL}
If for some constant $x>0$, $$\sum_{n} \mathbb P (b_{n} \leq xn) < \infty$$ then $\mu(e_{1}) \geq x$.
\end{lem} 
\begin{proof}
This  is a consequence of the Borel-Cantelli lemma as the time constant $\mu(e_{1})$ is also the limit of $b_{n}/n$ as $n$ goes to infinity \cite[Equation (1.13)]{Aspects}.
\end{proof}
\noindent We now fix $\delta >0$ and set 
\begin{equation}\label{def:definitionofxLB}
 x = \frac{(1-\delta)\log d}{2ad}.
 \end{equation}

\begin{lem}\label{lem:Sonice}
For any $\delta >0$, there exists $d$ large enough so that, for any fixed $n \in \mathbb N$ and  any $k\geq n$,
$$ \mathbb P (S_{k} \leq nx) \leq \bigg(\frac{(1+\delta)eanx}{k}\bigg)^{k}.$$ 
\end{lem}
\begin{proof}
Fix $\delta>0$. By  \eqref{eq:distLB}, we can choose $\epsilon_{0}$ small enough so that 
$$\mathbb P(\tau_{e} \leq t) \leq (1+\delta/2)at < 1$$ for all $t \in [0,\epsilon_{0}]$. Now, let $Y$ be any nonnegative random variable with density $f(y)=(1+\delta/2)a$ on $[0,\epsilon_{0}]$. For any $0\le \epsilon \le \epsilon_{0}$, the random variable $X_{1}:= \tau_{e} \indi_{\{ \tau_{e} < \epsilon\} } + \epsilon\indi_{\{ \tau_{e} \geq \epsilon\} }$ stochastically dominates $X_{2}:=Y\indi_{\{ Y < \epsilon\}} + \epsilon\indi_{\{ Y \geq \epsilon\} }$ as for all $t \in \mathbb R$,
$$ \mathbb P(X_{1} \leq t) \leq \mathbb P(X_{2}\leq t).$$ 
Therefore for any non-increasing function $\phi:\mathbb R \to \mathbb R$ 

\begin{equation}\label{eq:toverify}
\begin{split}
\mathbb E \phi(\tau_{e})\indi_{\{ \tau_{e} < \epsilon\} } + \phi(\epsilon) \Pro(\tau_{e} \ge  \epsilon)&\leq \mathbb \E \phi(Y)\indi_{\{ Y < \epsilon\}} +  \phi(\epsilon) \Pro(Y \ge \epsilon) \\
&\leq  \int_{0}^{\epsilon} \phi(y)f(y)d y + \phi(\epsilon).
\end{split}
\end{equation}
This implies that for any $0 \le \epsilon\le \epsilon_{0}$ and any $\gamma>0$, if we take $\phi(t)= e^{-\gamma t}$ we have
\begin{equation*}
\begin{split}
\mathbb P (S_{k} \leq nx) &\leq e^{n\gamma x} \mathbb E e^{-\gamma S_{k}} 
\leq  e^{n\gamma x} \bigg( \mathbb E e^{-\gamma \tau_{e}} \indi_{\{ \tau_{e}<\epsilon\}}+ e^{-\gamma \epsilon}\Pro(\tau_{e} > \epsilon)\bigg)^{k} \\
&\leq e^{n\gamma x}\bigg( \int_{0}^{\epsilon} e^{-\gamma y}f(y) dy+ e^{-\gamma \epsilon} \bigg)^{k} \leq e^{n\gamma x}\bigg(\frac{(1+\delta/2)a}{\gamma} + e^{-\gamma \epsilon} \bigg)^{k}.
\end{split}
\end{equation*}
Choose $d$ large enough so that $\sqrt{x} \le \min \{\epsilon_{0}, 1\}$ and such that for all $y \geq 1$ 
\begin{equation}\label{eq:inequality1}
y \exp\bigg(-\frac{y}{\sqrt{x}}\bigg) \leq \frac{\delta ax}{2}.
 \end{equation}
This is possible since the left side is monotonically decreasing in $y$ on $[1, \infty)$ if $x \le 1$. Hence one suffices to find $x \le \min \{\epsilon_{0}^{2}, 1\}$ such that 
\begin{equation}\label{eq:inequality2}
\exp\bigg(-\frac{1}{\sqrt{x}}\bigg) \leq \frac{\delta ax}{2} \quad \Leftrightarrow
\quad \sqrt{x}\log\frac{\delta a}{2} + 2\sqrt{x}\log \sqrt{x} \ge -1,
\end{equation} 
which is possible by choosing $d$ large enough and making each term above greater than $-1/3$.
Now set $\gamma = k/(nx)$. By taking $$\epsilon = \frac{k}{\gamma n \sqrt{x}} =\sqrt{x} \in [0, \epsilon_{0}],\quad y = \frac{k}{n} \geq 1,$$
we have by \eqref{eq:inequality1}
\begin{equation}
\begin{split}
 \mathbb P (S_{k} \leq nx) 
 &\leq \bigg( \frac{eax(1+\delta/2)}{y}+e^{-{\frac{y}{\sqrt{x}}+1}}\bigg)^{k}
\leq  \bigg( \frac{eax(1+\delta)}{y}\bigg)^{k},
\end{split}
\end{equation}
which proves the Lemma.
\end{proof}
We still need one combinatorial estimate that we take from \cite[(6.20)]{Aspects}.

\begin{lem} \label{lem:KestenCombinatorialBound}The number $\mathcal N_{k,n}$ of lattice paths in $\mathbb Z^{d}$ from $0$ to $H_{n}$ of $k$ steps is at most 
$$(2d)^{k}\min \bigg (1, \exp\big (-n\rho + \frac{k}{d}(\cosh \rho - 1) \big) \bigg)$$
for any $\rho\geq 0$. 
\end{lem}

\begin{proof}[Proof of Proposition \ref{prop:LB}] For $0<\delta<1$ fixed choose $x$ as in \eqref{def:definitionofxLB}. 

We will use the union bound 
$$\Pro\big(b_{n} \leq nx\big) \leq \sum_{k=n}^{\infty} \mathcal N_{k,n} \Pro(S_{k}\leq nx).$$
Set $M=4enaxd = 2en(1-\delta)\log d$. Using Lemma \ref{lem:KestenCombinatorialBound} with $\rho = \log(\frac{2dn}{k})$ for $n\leq k\leq M$ and Lemma \ref{lem:Sonice} we have
\begin{equation}\label{eq:HowAreYou}
\begin{split}
\mathbb P\big(b_{n} \leq nx\big) \leq \sum_{n\leq k \leq M}  \bigg(\frac{ke}{2dn}\bigg)^{n}\bigg(\frac{2d(1+\delta)eanx}{k}\bigg)^{k}+ \sum_{k>M} \bigg(\frac{2d(1+\delta)eanx}{k}\bigg)^{k}
\end{split}
\end{equation}
Choose $\delta \le 1/2$, the second sum in the right side of \eqref{eq:HowAreYou} is bounded above by
$$ \bigg(\frac{1+\delta}{2}\bigg)^{M} \frac{1}{1-\frac{1+\delta}{2}}\le 4 \bigg(\frac{3}{4}\bigg)^{2en(1-\delta)\log d}$$
which is summable in $n$. 
On the other hand, if we write $z = k/n$, a little algebra implies that the first sum in \eqref{eq:HowAreYou} is bounded above by 
\begin{equation}\label{eq:almostthere}
M \bigg[ \frac{e}{2d} \max_{z} \bigg( z^{1-z} [(1-\delta^{2})e \log d]^{z} : 1 \leq  z \leq 2e (1-\delta) \log d \bigg) \bigg]^{n}.
\end{equation}
The term inside the large square bracket above is bounded by 
\begin{equation}\label{eq:almostthere2}
 2e^{2} (1-\delta) \frac{\log d}{2d} \max_{z} \bigg( \left[\frac{(1-\delta^{2})e \log d}{z}\right]^{z} : 1 \leq  z \leq 2e (1-\delta) \log d \bigg) 
 \end{equation}
As for any $c>0$, the function $f(z)=(c/z)^{z}$ has a maximum value equal to $e^{{c/e}}$ on $0 \leq z \leq c$,  we obtain that \eqref{eq:almostthere2} is bounded above by 
\begin{equation}\label{Havetoaddthisonetoo}
  e^{2} (1-\delta) \frac{\log d}{d} \exp((1-\delta^{2})\log d) \le e^{2} d^{-\delta^{2}} \log d.
  \end{equation}
Given the choice of $\delta$, this term is strictly bounded above by $1$ for $d$ large enough,
which turns \eqref{eq:almostthere} summable in $n$. Now the proposition follows from Lemma \ref{lem:basiclemmaBL} with $x= \frac{(1-\delta)\log d}{2ad}$ and sending $\delta$ to $0$.
\end{proof}

We now end this subsection with the proof of Theorem \ref{thm:Thisisthefirsttheorem} in the case $a=\infty$ and $a=0$. In this case, assumption \eqref{eq:todayisFriday} implies that for any $M>0$ it is possible to find $\epsilon >0$ such that for any $x<\epsilon$, $\mathbb P (\tau_{e} \leq x) \geq x M$. Let $Y_{M}$ be a random variable with density $f(y) = My$ on $[0,\epsilon]$ and such that  for any $t\in \mathbb R$ 
$$ \Pro (Y_{M} \leq t) \leq \Pro(\tau_{e} \leq t).$$
As $\Pro(Y_{M}\leq \epsilon) \leq \Pro(\tau_{e}\leq \epsilon)$, this random variable can be constructed by simply choosing a non-decreasing function on $[\epsilon,\infty)$ that has limit $1$ at $\infty$ and is bounded above by $\Pro(\tau_{e} \leq t)$.
This way, we can use the comparison theorem of van den Berg-Kesten\cite[Theorem 2.13]{MR1202515} to obtain for any $d\geq 2$, $\mu(e_{1}) \leq \mu^{Y_{M}}(e_{1})$, where $\mu^{Y_{M}}$ is the time constant for FPP in $\mathbb Z^{d}$ with passage times distributed according to $Y_{M}$. Since $Y_{M}$ satisfies the hypothesis of Theorem  \ref{thm:Thisisthefirsttheorem} with $a=M$, we get 

$$\limsup_{d \to \infty} \frac{\mu(e_{1})d}{\log d } \leq  \lim_{d\to \infty} \frac{\mu^{Y_{M}}(e_{1})d}{\log d} = \frac{1}{2M}.$$
Taking $M$ to infinity gives us the desired result.

The case $a=0$ is similar. For any $m>0$ we  construct a random variable $Y_{m}$ that dominates $\tau_{e}$ as $ \Pro (Y_{m} \leq t) \geq  \Pro (\tau_{e} \leq t)$ and satisfies \eqref{eq:dist2} with $a=m$. Van den Berg-Kesten comparison theorem combined with the result for $a>0$ implies 
$$\liminf_{d \to \infty} \frac{\mu(e_{1})d}{\log d } \geq  \lim_{d\to \infty} \frac{\mu^{Y_{m}}(e_{1})d}{\log d} = \frac{1}{2m}.$$
The result follows by taking $m$ to zero.

\section{Application to the limit shape}\label{sec:Kuppo}
In this section, we will exclude certain candidates of possible limit shapes in high dimension, including the Euclidean ball. The method here is the same as the one used by Kesten \cite{Aspects}. We will compare the time constant $\mu(e_{1})$ in the $e_{1}$ direction with time constant $\mu^{\ast}$ in the diagonal direction. Here $\mu^{\ast}$ is defined as 
\[
\mu^{\ast} \mathdef \lim_{n\to \infty}\frac{T(\mathcal J_{n})}{n}, \quad \text{a.s.,}
\] 
where $\mathcal J_{n}$ is the hyperplane defined by $\mathcal J_{n} \mathdef \{(x_{1}, \ldots, x_{d}): x_{1} + x_{2}+\cdots+x_{d}=n\sqrt{d}\}$ and the limit hold in $L^{1}$. Without any loss in what follows, we slightly abuse our notation by taking $\sqrt{d}$ as the smallest integer greater than the square root of $d$. It has been shown in \cite{MR2753301} that, when $\tau_{e}$ follows a standard exponential distribution, then for $d\ge 2$,
\begin{equation}\label{eq:lowebounddiag}
\mu^{\ast} \ge \frac{\sqrt{\alpha^{2}_{\ast}-1}}{2\sqrt{d}} \geq \frac{0.3313}{\sqrt{d}},
\end{equation}
where $\alpha_{\ast}$ is the non null solution of $\coth \alpha= \alpha$. Recently, still under the assumption of exponential passage times, Martinsson \cite{Martinsson2015Arxiv} proved a matching upper bound establishing  
$$\lim_{d \to \infty} \sqrt{d} \mu^{\ast} = \frac{\sqrt{\alpha^{2}_{\ast}-1}}{2}. $$ 

\subsection{A lower bound in the diagonal direction}
We start by showing that the lower bound \eqref{eq:lowebounddiag} is also true under our setting in the large $d$ limit:
 \begin{thm} \label{thm:mustar}Suppose the edge weight distribution satisfies \eqref{eq:dist2} and $\mu^{\ast} = \mu^{\ast}_{d}$ is the time constant in the diagonal direction as defined above. Then 
 \[
 \liminf_{d\to \infty} \sqrt{d} \mu^{\ast} \ge \frac{\sqrt{\alpha_{\ast}^{2}-1}}{2a}, 
 \]
\end{thm}
\begin{rem} The result also holds under \eqref{eq:distLB}, which is more general than \eqref{eq:dist2} and the proof is a slight modification of that in Section \ref{sec:lb} and \cite{MR2753301}. However, we state Theorem \ref{thm:mustar} and provide its proof under \eqref{eq:dist2} for two purposes: (i) to quantify the difference between $\mu^{\ast}$ and $\frac{\sqrt{\alpha_{\ast}^{2}-1}}{2a\sqrt{d}}$, for $d$ large, and (ii) to compare the lower bound of $\mu^{\ast}$ with the upper bound of $\mu$ for finite but large $d$, under the same set of conditions. 
\end{rem}
\begin{proof} For $\delta \in (0,1)$, we may always choose $d$ large enough such that 
\begin{equation}
\label{eq:relmustar1}
 d \ge \frac{\alpha_{\ast}^{2}-1}{4a^{2}} \max\left \{\epsilon_{0}^{-4}, e^{8C/\delta}, 1\right\}.
\end{equation}
Recall that $\epsilon_{0}$ is the right endpoint of the interval $[0, \epsilon_{0}]$ on which the distribution of $
\tau_{e}$ satisfies \eqref{eq:dist2} with constants $a$ and $C$. 
We set $x = \frac{(1-\delta) \sqrt{\alpha_{\ast}^{2}-1}}{2a\sqrt{d}}$. Note that $\sqrt{x} \in [0,  \min\{\epsilon_{0}, 1\}]$ due to \eqref{eq:relmustar1}. Then for any fixed $n \in \mathbb N$ and any $k\ge n\sqrt{d}$, we can repeat the computation before \eqref{eq:inequality1} to obtain 
\begin{equation}
\label{eq:noname}
\Pro(S_{k}\le nx) \le \left( \frac{(1+\delta) eanx}{k}\right)^{k}.
\end{equation}
To see this, one can just follow the proof of Lemma \ref{lem:Sonice} with $y = \frac{k}{n\sqrt{d}}$ and $\gamma = k/(nx\sqrt{d})$, provided \eqref{eq:inequality2} holds for our choice of $x$ and any $y \ge 1$, i.e., 
\begin{equation}
\label{eq:relmustar2}
\frac{\sqrt{2a}}{(\alpha_{\ast}^{2}-1)^{\frac{1}{4}}}d^{\frac{1}{4}} - \frac{1}{2} \log d \ge \log \frac{4}{\delta(1-\delta) \sqrt{\alpha_{\ast}^{2}-1}}.
\end{equation}
This is always possible by choosing $d$ large enough.
Next, we use the upper bound for the number $D_{k}^{(n)}$ of self-avoiding walks of length $k$ from $0$ to $\mathcal J_{n}$ from Lemma 3 of \cite{MR2753301}:
\begin{align*}
D_{k}^{(n)} &\le (2d)^{k}\frac{n\sqrt{d}}{k} \binom{k}{(k+n\sqrt{d})/2}2^{-k}\\
& \le \sqrt{\frac{1}{ n \sqrt{d}}} \left(\frac{2dy}{(y+1)^{(y+1)/(2y)} (y-1)^{(y-1)/2y}}\right)^{k}.
\end{align*}
Hence, we have
\begin{align*}
\Pro(T(\mathcal J_{n}) \le nx) &\le \sum_{k \ge n\sqrt{d}}D_{k}^{(n)}\Pro(S_{k}\le nx)\\
&\le  \sqrt{\frac{1}{ n \sqrt{d}}} \sum_{k \ge n\sqrt{d}} \left(\frac{2\sqrt{d}(1+\delta) eax}{(y+1)^{(y+1)/(2y)} (y-1)^{(y-1)/2y}}\right)^{k}\\
&\le  \sqrt{\frac{1}{ n \sqrt{d}}}  \sum_{k \ge n\sqrt{d}}\left(\frac{(1-\delta^{2})e\sqrt{\alpha_{\ast}^{2}-1}}{\inf_{y \ge 1} \{ (y+1)^{(y+1)/(2y)} (y-1)^{(y-1)/2y} \} }\right)^{k}.
\end{align*}
Note that $\inf_{y \ge 1} \{ (y+1)^{(y+1)/(2y)} (y-1)^{(y-1)/2y}\} = e\sqrt{\alpha_{\ast}^{2}-1}$, and hence $\Pro(T(\mathcal J_{n})\le nx)$ decays exponentially. The conclusion follows from the Borel-Cantelli lemma.
\end{proof}
\noindent The proof above implies the following quantitative lower bound. 
\begin{corollary} 
The bound 
\begin{equation}\label{eq:mustarlowbd}
\mu^{\ast} \ge \frac{\sqrt{\alpha_{\ast}^{2}-1}}{2a\sqrt{d}} (1-\delta)
\end{equation}
holds for all $\delta \in (0, 1)$ and $d \ge 1$ such that   \eqref{eq:relmustar1} and \eqref{eq:relmustar2} are satisfied.
\begin{proof}[Proof of Theorem \ref{thm:YellowChocobo}] Firstly, choose $d_{1}$ sufficiently large such that  \eqref{eq:relmustar1} and \eqref{eq:relmustar2} are satisfied with $\delta = 1/2$, for all $d > d_{1}$. This means $\mu^{\ast} \ge \frac{1}{2} \frac{\sqrt{\alpha_{\ast}^{2}-1}}{2a\sqrt{d}} $ for all $d > d_{1}$. Secondly, due to Theorem \ref{thm:Thisisthefirsttheorem}, we can choose $d_{2}$ large enough such that $\mu(e_{1}) \le \frac{\log d}{ad}$ for all $ d > d_{2}$. Choose $d_{3}$ large enough such that $\frac{1}{2} \frac{\sqrt{\alpha_{\ast}^{2}-1}}{2a\sqrt{d}}  > \frac{\log d}{ad}$ for all $d > d_{3}$. Putting $d_{0} = \max(d_{1}, d_{2}, d_{3})$, we have for all $d > d_{0}$,
\[
\mu(e_{1}) \le \frac{\log d}{ad} < \frac{1}{2} \frac{\sqrt{\alpha_{\ast}^{2}-1}}{2a\sqrt{d}} \le \mu^{\ast},
\]
which means the intersection of $\mathcal B$ with the line $\ell_{d} =\{\lambda(1,1,\ldots, 1): \lambda \in \mathbb R\}$ is strictly contained in the Euclidean ball $\mathsf B$. This shows that $\mathcal B \neq \mathsf B$ and $ \mathcal B \subsetneq  \mathsf C$. 

Note that to show that the limit shape is not equal to the $d$-dimensional diamond $\mathsf D$, it suffices to show that 
\begin{equation}\label{Hahahaha or najifnduifs}
\mu^{*} < \sqrt{d} \mu(e_{1}).  
\end{equation}
The term $\sqrt{d}$ appears due to our choice of normalization of the hyperplane $\mathcal J_{n}$. 
To prove \eqref{Hahahaha or najifnduifs}, construct a path $\gamma^{n} = \{\< v_{i},v_{i+1}\>\}_{i=1,\ldots,n\sqrt{d}}$ from $0$ to $\mathcal J_{n}$ such that 
\begin{enumerate}
\item  For every $i =1, \ldots, n\sqrt{d}$, the $i$-th edge $\< v_{i},v_{i+1}\>$ of $\gamma^{n}$ is always in the positive direction, that is, $v_{i+1}-v_{i} \in \{e_{1}, \ldots,e_{d}\}$. 
\item $\<v_{i},v_{i+1}\>$ is chosen as the edge with the smallest weight coming out of $v_{i}$, that is, $$\tau_{\<v_{i},v_{i+1}\>} = \min_{k=1,\ldots,d} \big\{ \tau_{ \< v_{i},v_{i}+e_{k}\>}\big\}.$$
\end{enumerate}
If we let $$Y=\min \big\{ t_{1}, \ldots, t_{d} \big\},$$ where the $t_{i}$'s are independent copies of $\tau_{e}$, we see by construction,

\begin{equation}\label{HelloSI!HowAreYOU}
\mathbb E T(\gamma^{n}) = n\sqrt{d}\, \mathbb E Y.
\end{equation}
Hence,
$$ \mu^{*} =  \lim_{n\to \infty}\frac{\mathbb E T(\mathcal J_{n})}{n} \leq \lim_{n\to \infty} \frac{\mathbb E T(\gamma^{n})}{n} = \sqrt{d}\, \E Y. $$ 
Our assumptions on the distribution of $\tau_{e}$ imply the existence of a positive constant $c$, independent of $d$, such that $\mathbb E Y \leq c/d$ for all $d\geq 2$. Indeed, by hypothesis \eqref{eq:dist2} and the fact that $a>0$, we can find $\delta, \epsilon >0$ so that $$\mathbb P(\tau_{e}>x) = 1 - \mathbb P (\tau_{e} \leq x) \leq 1 - xa (1-\epsilon) < 1$$ for all $x\in [0,\delta]$ and such that $\Pro(\tau<\delta) <1$. Choosing a costant $m >\mathbb E \tau_{e}$, we write

\begin{equation}
\begin{split}
 \E Y = \int_{0}^{\infty} \mathbb P(\tau_{e} >t)^{d} dt &\leq \int_{0}^{\delta} (1-at(1-\epsilon))^{d} dt + \int_{\delta}^{m} \mathbb P(\tau_{e} >t)^{d} dt +\int_{m}^{\infty} \mathbb P(\tau_{e} >t)^{d} dt \\
 &\leq \frac{1}{d+1}\frac{1}{a(1-\epsilon)} + (m-\delta) \mathbb P(\tau_{e} >\delta)^{d} + \bigg(\frac{\E \tau_{e}}{m}\bigg)^{d}\frac{m}{d-1},
\end{split}
\end{equation}
where in the first integral we used hypothesis \eqref{eq:dist2} and in the last one we used Markov's inequality. As $\mathbb P (\tau_{e}>\delta) <1$, we can find $c>0$ such that all three terms are bounded above by $c/d$, for all $d\geq 2$. 
 Combining this with Proposition \ref{prop:LB}, we have for any $d$ large enough 
$$  \mu^{*} \leq \frac{c}{\sqrt{d}} \leq \frac{\log d}{3a \sqrt{d}} < \sqrt{d} \mu(e_{1}),$$
proving \eqref{Hahahaha or najifnduifs}.

\end{proof} 

\end{corollary}

\section*{Appendix } In this section, we show how to compute a quantitative upper bound of $\mu(e_{1})$. This bound will only depend on $d$, $\mathbb E \tau_{e}$ and  $a,\epsilon_{0},C$ satisfying \eqref{eq:dist2}.  

We proceed as follows. First, slightly modifying \eqref{eq:finalcal} in Section \ref{sec:ub},  we notice that for any $A>0, \delta \in (0,1)$ and $B\neq 0$ satisfying  
\begin{equation}\label{eq:conditionforquant}
\E (N_{n, x}^{2}) \le A (\E N_{n,x})^{2} \quad \text{ and  } \quad y := \frac{B\delta \log d}{2ad} \leq \min \{\epsilon_{0}, e^{-C}\},
\end{equation}
and any $\eta >0$,
\begin{align}
\label{eq:muupbd}
\mu(e_{1}) \leq \frac{\log d}{2ad}\left(B\delta + \frac{1}{1-\delta} \right)
 + \left(1-\frac{B\delta\log d}{2Ad}\left[1-\frac{C}{|\log y|}\right]\right)^{\left \lfloor \frac{d}{\delta^{1+\eta}\log d}\right \rfloor}\E \tau_{e}. 
\end{align}
Letting $\Upsilon(A,B,\delta,\eta)$ be the right side of  \eqref{eq:muupbd}, we obtain

\begin{equation}\label{mycoffeemugissocute}
\mu(e_{{1}}) \leq \inf_{(\delta,A,B,\eta)} \Upsilon(A,B,\delta,\eta),
\end{equation}
where the infimum is taken over all $\eta >0$ and $A,B,\delta$ satisfying \eqref{eq:conditionforquant}.

As we will see, the main task here is to find any (but preferably the smallest) $A$ that satisfies $\eqref{eq:conditionforquant}$. In Section \ref{sec:kdosadkasokpojigorhughrurejioef}, we found such a number for any given $\delta \in (0,1)$ and $\eta >0$. Indeed, we saw that we can choose any $A=A(\delta, \eta)$ such that 
\begin{equation}
\label{eqn:uppA}
A \geq 1 + e^{f_{a, C}(\delta, d) g_{\eta}^{2}(\delta, d)} \left( 1+ o(p^{-1/2})\right) \sum_{l=0}^{n-1}\left(\frac{1-\delta}{1-\frac{1}{\delta^{1+\eta}\log d}}\right)^{l}+ \frac{1}{\E N_{n, x}},
\end{equation}
where upper bounds of $1/\E N_{n, x}$, $f_{a,C}(\delta, d)$, and $g_{\eta}(\delta, d)$ can be recovered from \eqref{eq:1stmom}, \eqref{eq:Aconst1} and \eqref{eq:Aconst2}, respectively, once we are given $\delta, \eta, d$ and the parameters $a, C$ in \eqref{eq:dist2} for the distribution of $\tau_{e}$. The issue now is to control the $o(p^{-1/2})$ term, where $p=d - \lfloor \frac{d}{\delta^{1+\eta}\log d}\rfloor >  0$. This will be done in the next section of the appendix. In the last section, we provide a few specific examples of these computations.

\subsection*{A refinement of Lemma \ref{lem:main}}
To get a quantitative estimate of the lower bound of $A$ using \eqref{eqn:uppA}, we need to get control of the $o(p^{-1/2})$ term.  This term comes exactly from the estimate on the  probabilities in Lemma \ref{lem:main}. We will control these probabilities by computing the combinatorics in the ``bubble'' argument more precisely. We start from the two cases defined in \eqref{eq:bubblecase1} and \eqref{eq:bubblecase2}.\br
For the first case, it is not difficult to see from the proof of Lemma 3.4(i) that,
\[
\Pro( \text{Case (i)}) \le \left( \frac{1}{2p}\right)^{l}\left \{ 1 + \frac{(n-l-1)^{2}}{p}+2p\sum_{K=2}^{l}\frac{\binom{l-1}{K-1} \left[ \binom{n-1}{K} \cdot K! \right]^{2} }{p^{K}}\right\} = \left( \frac{1}{2p}\right)^{l} \big[ 1 + (\text{I})\big],
\]
where the right side is obtained by replacing $n$ in Lemma 3.4(i) by $(n-1)$ and following the combinatorics there.\br
For the second case, we provide more detail here.  As before, we let $\tilde \gamma$ and $\tilde \gamma'$ be the paths obtained by removing the last step of $\gamma$ and $\gamma'$, respectively. We use $B_{K}(\tilde \gamma, \tilde \gamma')$ to denote the event that the overlapping edges in $\tilde \gamma$ and $\tilde \gamma'$ are clustered in $K$ segments. We compute $K=1, 2, 3$ case separately, whereas all $K \ge 4$ cases are considered together.\br
When $K = 1$ and the overlapping segments in $\gamma$ and $\gamma'$ align in the same direction, there can be either one bubble or two bubbles. The ``one-bubble'' diagram is in Figure 4.
Hence 
\begin{align*}
\Pro(\text{Case (ii)}; B_{1}(\tilde \gamma, \tilde\gamma')) &\le 2\cdot \left (\frac{1}{2p}\right )^{l-1} \left( \frac{2(n-l)-2}{2p}\right)^{2}  \\
&+  \left (\frac{1}{2p}\right )^{l-1}  \sum_{m_{1}, m_{1}'=1}^{n-l-1} \left( \frac{m_{1} + m_{1}'-2}{2p}\right)^{2} \left( \frac{2n-2l-m_{1} - m_{1}'-2}{2p}\right)^{2} \\
& + \left (\frac{1}{2p}\right )^{l-1} (n-l-1)^{2}  \left (\frac{1}{2p}\right )^{2} = \left( \frac{1}{2p}\right)^{l}\cdot (\text{II}) 
\end{align*}
When $K=2$ and there is only one bubble, both overlapping segments have to align in the same direction and attach to the start and endpoint of $\tilde \gamma$ and $\tilde \gamma'$. This diagram occurs with probability no more than 
\[
(l-2) \left (\frac{1}{2p}\right )^{l-1} \left(\frac{2n-2l-2}{2p}\right)^{2},
\]
where $(l-2)$ is the number of ways of grouping the $(l-1)$ overlapping edges into two consecutive overlapping segments, $\left(\frac{2n-2l-2}{2p}\right)^{2}$ is the factor for the middle bubble, applying \eqref{eq:fact2}. We can also estimate the probabilities for other types of alignments for $K=2$ case, which gives  
\begin{align*}
\Pro(\text{Case (ii)}; B_{2}(\tilde \gamma, \tilde\gamma')) \le& \left (\frac{1}{2p}\right )^{l-1} (l-2)\left[6\cdot \left(\frac{n-l-1}{2p}\right)^{2}  + \right. \\
& \left.   \frac{2(n-l-1)^{2}(n-l-2)^{4}}{(2p)^{4}} + \frac{4(n-1)^{2}(n-l)^{2}}{(2p)^{3}}\right]\\
= &\left( \frac{1}{2p}\right)^{l}\cdot (\text{III}). 
\end{align*}
Following the same strategy, we can compute the probabilities for the $K=3$ case
\begin{align*}
\Pro(\text{Case (ii)}; B_{3}(\tilde \gamma, \tilde\gamma'))& \le \binom{l-2}{2} \left (\frac{1}{2p}\right )^{l-1} \left [\frac{2(n-l)^{2}}{(2p)^{2}}  + \frac{8(n-1)^{4}(n-l)^{2}}{(2p)^{3}}\right] \\
&=\left( \frac{1}{2p}\right)^{l}\cdot (\text{IV}) ,
\end{align*}
and the union of all $K \ge 4$ case
\[
\sum_{K=4}^{l}\Pro(\text{Case (ii)}; B_{K}(\tilde \gamma, \tilde\gamma')) \le  \left (\frac{1}{2p}\right )^{l-1} \cdot 2p\sum_{K=4}^{l} \frac{\binom{l-2}{K-1} \left[ \binom{n-1}{K} \cdot K! \right]^{2} }{p^{K}} = \left( \frac{1}{2p}\right)^{l}\cdot (\text{V}). 
\]
Combining everything above gives a quantitative upper bound for $\Pro(\text{Case (ii)})$. The quantitative upper bound can be reduced even further by computing $\Pro(\text{Case (ii)}; B_{K}(\tilde \gamma, \tilde\gamma')) $ separately for more $K$'s, but we stop here. \br
Therefore, $o(p^{-1/2})$ is no more than summing (I) through (V).

\subsection*{Special Case: Exponential Distribution}

\subsubsection*{Limit shape is not a $d$-dimensional cube for $d\geq 269,000.$}
For the special case of $\tau_{e}$ following a standard exponential distribution, we have $a=1, C=o(1)$, and $\E \tau_e = 1$. Moreover, $\mu^{\ast} \ge \frac{\sqrt{\alpha_{\ast}^{2}-1}}{2\sqrt{d}} $ for all $d \ge 2$. To see this, one can either check \cite{MR2753301} or go through the proof of Theorem \ref{thm:mustar}, using the fact that \eqref{eq:noname} holds for $\delta = 0$.\br
To estimate $\mu(e_{1})$, instead of using Lemma \ref{lem:kesten8.8}, we can compute $\Pro(S_{n}\le x)$ explicitly in this case,
\[
\Pro(S_{n}\le x) =\int _{0}^{x} \frac{y^{n-1}e^{-y}}{(n-1)!}dy\mathdef \boldsymbol \gamma(n, x)
\]
where $\boldsymbol \gamma(n, x)$ is the cumulative distribution function of a gamma distribution with shape parameter $n$ and scale parameter $1$. 
This allows us to replace the ratio 
$
\frac{\Pro(S_{n-l}\le x)}{\Pro(S_{n}\le x)}$ by $\frac{\boldsymbol \gamma(n-l, x)}{\boldsymbol \gamma(n, x)}
$. 
We also use $\boldsymbol \gamma(n, x)$ to estimate $1/\E N_{n, x}$
\[
\frac{1}{\E N_{n,x}} = \frac{1}{ |\mathcal P_{n}|\cdot \Pro(S_{n} \le x)}\le \frac{1}{[2p-1-\log(2p-1)]^{n-1}\boldsymbol \gamma(n, x)}.
\]

Therefore, following exactly the same computation in Section \ref{sec:ub}, we obtain that we can choose any 
\begin{equation}\label{dasdasdasdasdajsiojuiorhtuihnreutghnerugbnrfiughroeugyiohrenwgfhrygfyngdi}
A \ge 1+ \frac{1}{\E N_{n, x}} + e^{2g_{\eta}(\delta, d)}  \sum_{l=1}^{n-1} \frac{\boldsymbol \gamma(n-l, x)}{\boldsymbol \gamma(n, x)}\left[ \Pro(\text{Case (i)}) + \Pro(\text{Case (ii)})\right],
\end{equation}
to obtain
\begin{equation}\label{dajsdioajdioas}
\mu(e_{1}) \le \inf_{(\delta,B,\eta)} \bigg\{ \frac{\log d}{2d}\left(B\delta + \frac{1}{1-\delta} \right)
 + \left(1-\frac{1-e^{-  \frac{B\delta \log d}{2d} }}{A}\right)^{\left \lfloor \frac{d}{\delta^{1+\eta}\log d}\right \rfloor} \bigg\}.
\end{equation}
For instance, plugging in $d=268,337, \eta = 10^{-3}, \delta=0.764$ and $B=23.85$, we can take $A = 1.20$ and $\mu(e_{1}) \le 0.000638$ and $\mu^{\ast} \ge0.000639 > \mu(e_{1})$. 
We checked that for all $d \in (268337, 10^{6})$, $  \frac{\sqrt{\alpha_{\ast}^{2}-1}}{2\sqrt{d}} $ is greater than the right side of \eqref{dajsdioajdioas}. The case $d \geq 10^{6}$ follows directly from Kesten\cite[Remark 8.5]{Aspects}. 

\subsubsection*{Limit shape is not a $d$-dimensional diamond for any $d \geq 110$.}
Recall that $Y \mathdef \min \{t_{1}, \cdots, t_{d}\}$, where $t_{i}$'s are i.i.d. copies of $\tau_{e}$.  For the standard exponential distribution we have $\mathbb E Y = d^{-1}$ in \eqref{HelloSI!HowAreYOU} so $\mu^{*} \leq d^{-1/2}$ for any $d\geq 2$. At the same time, we can take $\delta = 0$ in Lemma \ref{lem:Sonice}, which implies that the exponent of $\delta$ in \eqref{eq:almostthere} changes from $2$ to $1$. This leads that \eqref{eq:HowAreYou} and \eqref{Havetoaddthisonetoo}
are summable for any $\delta \in (0,1)$ and all $d$ such that 
\begin{equation}\label{Thismustbethelastequationinthepaper!!}
e^{2} d^{-\delta} (1-\delta)\log d < 1,
\end{equation}
which gives, by Lemma \ref{lem:basiclemmaBL},  the bound $\sqrt{d}\mu(e_{1}) \geq \frac{(1-\delta)\log d}{2\sqrt{d}}$ for this choice of $\delta$ and $d$. This implies that the limit shape is not a $d$-dimensional cube (as $\mu^{*} < \sqrt d \mu(e_{1})$) for any $d$ that satisfies $$ d > \exp\bigg( \frac{2}{1-\delta}\bigg)$$ and 
\eqref{Thismustbethelastequationinthepaper!!}. Choosing $\delta=(2+\log 2)/(4+\log 2)$ we see that the equations above are satisfied for any $d \geq 110$.

If $\tau_{e}$ is not exponentially distributed and we only have $\mathbb E Y \le d^{-1}$, for instance, when $\tau_{e} \sim \mathcal U[0, 1]$, we must keep the exponent of $\delta$ in \eqref{eq:almostthere} equal to $2$ and the choice of $\delta =0.669$ excludes the $d$-dimensional diamond for $d \geq 416$.

\bibliographystyle{amsplain}
\bibliography{notes_ref}
\vspace{1cm}
\end{document}